\documentclass[12pt]{amsart}
\usepackage{graphicx}
\usepackage{hyperref}
\usepackage{setspace}
\usepackage{enumerate}
\usepackage{amsmath,amsfonts,amssymb,amscd}

\newtheorem{thm}{Theorem}[section]
\newtheorem{cor}[thm]{Corollary}
\newtheorem{lem}[thm]{Lemma}

\newtheorem{example}[thm]{Example}
\theoremstyle{definition}

\theoremstyle{remark}

\numberwithin{equation}{section}

\newcommand{\ord}{\textup{ord}\,}

\newcommand{\res}{\textup{res}\,}

\begin{document}

\title[Diophantine approximation of polynomials over \text{$\mathbb{F}_q[t]$}]{Diophantine approximation of polynomials over $\mathbb{F}_q[t]$ \\ satisfying a divisibility condition }

\author{Shuntaro Yamagishi}
\address{Department of Pure Mathematics \\
University of Waterloo \\
Waterloo, ON\\  N2L 3G1 \\
Canada}
\email{syamagis@uwaterloo.ca}
\indent


\begin{abstract}
Let $\mathbb{F}_q[t]$ denote the ring of polynomials over $\mathbb{F}_q$, the finite field of $q$ elements.
We prove an estimate for fractional parts of polynomials over $\mathbb{F}_q[t]$
satisfying a certain divisibility condition analogous to that of intersective polynomials
in the case of integers. We then extend our result to consider linear combinations
of such polynomials as well.
\end{abstract}

\subjclass[2010]
{11J54, 11T55}

\keywords{fractional parts of polynomials; function field; intersective polynomials}

\maketitle

\section{Introduction}
In 1927, Vinogradov \cite{V} proved the following result, confirming a conjecture of Hardy and Littlewood
\cite{HL}. Let $\| \cdot \|$ denote the distance to the nearest integer.
\begin{thm} \label{thm vinogradov}
For every positive integer $k$, there exists an exponent $\theta_k > 0$
such that
$$
\min_{1 \leq n \leq N} \| \alpha n^k \| \ll _k N^{- \theta_k}
$$
for any positive integer $N$ and real number $\alpha$.
\end{thm}
A brief history and introduction to the topic is given in \cite[Section 1]{LS}, which we paraphrase here.
Vinogradov showed that one could take $\theta_k = \frac{k}{k 2^{k-1} + 1 } - \epsilon$ for any
$\epsilon > 0$. In particular, one can take $\theta_2 = 2/5 - \epsilon$. Heilbronn \cite{H} improved
this to $\theta_2 = 1/2 - \epsilon$. The best result to date is due to Zaharescu \cite{Z}, who showed
we can take  $\theta_2 = 4/7 - \epsilon$, though his method is not applicable to higher powers.
It is an open conjecture that we can choose $\theta_2$ (and more generally $\theta_k$) to be $1 - \epsilon$.

Natural generalizations of Vinogradov's result have been made. Davenport \cite{D} obtained
an analogue of Theorem \ref{thm vinogradov} when $n^k$ is replaced by a polynomial
$f(n)$ of degree $k$ without a constant term (the corresponding bound being uniform in the coefficients
of $f$ and depending only on $k$). Notably, the best bound is due to Wooley, who showed that we can
choose $\theta_k = \frac{1}{4k(k-2)} - \epsilon$ for $k \geq 4$, as a consequence of his recent breakthrough
\cite{W} on Vinogradov's mean value theorem. We note that Vinogradov's result has also been generalized to
simultaneous approximation, where we consider multiple polynomials at once. However, we focus
on the single polynomial case in this paper and we refer the reader to \cite[Section 1]{LS} for more information on
simultaneous approximation.

In contrast, L\^{e} and Spencer put more emphasis on the qualitative side of these problems in \cite{LS}.
They were interested in generalizing Theorem \ref{thm vinogradov} in the following manner. For instance,
is it possible to replace $n^k$ in Theorem \ref{thm vinogradov} with a polynomial $h \in \mathbb{Z}[x]$?
That is, for which polynomials $h \in \mathbb{Z}[x]$ do we have
$$
\min_{1 \leq n \leq N} \| \alpha h(n) \| \ll_h N^{-\theta}
$$
for some $\theta = \theta(h)$, uniformly in $\alpha$ and $N$?
By the result of Davenport \cite{D} mentioned in the previous paragraph, this is the case if $h$ is without a constant term,
but apparently these are not all the polynomials satisfying this property.
By considering $\alpha = 1/q$, we see that in order for such a bound to exist,
$h$ must have a root modulo $q$ for every $q \in \mathbb{Z}^+$.
Clearly, this condition is satisfied by polynomials without constant terms.
L\^{e} and Spencer proved that this condition is also sufficient. 
\begin{thm} \cite[Theorem 3]{LS} \label{thm LS 1}
Let $h$ be a polynomial in $\mathbb{Z}[x]$ with the property that for
every $q \not = 0$, there exists $n_q \in \mathbb{Z}$, $0 \leq n_q < q$, such that $h(n_q) \equiv 0 \ (\text{mod } q )$.
Then there is an exponent $\theta > 0$ depending only on the degree of $h$
such that
$$
\min_{1 \leq n \leq N} \| \alpha h(n) \| \ll_h N^{- \theta}
$$
for any positive integer $N$ and real number $\alpha$.
\end{thm}

Our goal in this paper is to consider analogous problems of qualitative nature over $\mathbb{F}_q[t]$, where $\mathbb{F}_q$ is
a finite field of $q$ elements, taking the approach of L\^{e} and Spencer in \cite{LS}.
However, before we can state our results we need to introduce notation, some of which
we take from the material in \cite[Section 1]{LL}.
We denote the characteristic of $\mathbb{F}_q$, a positive prime number, by ch$(\mathbb{F}_q) = p$.
Let $\mathbb{K} = \mathbb{F}_q(t)$ be the field of fractions of the polynomial ring $\mathbb{F}_q[t]$.
For $f/g \in \mathbb{K}$, we define the norm $| f/g | = q^{\deg f - \deg g}$ (with the convention that $\deg 0 = - \infty$).
The completion of $\mathbb{K}$ with respect to this norm is $\mathbb{K}_{\infty} = \mathbb{F}_q((1/t))$,
the field of formal Laurent series in $1/t$. In other words, every element $\alpha \in \mathbb{K}_{\infty}$
can be written as $\alpha = \sum_{i= - \infty}^{n} a_i t^i$ for some $n \in \mathbb{Z}$ and $a_i \in \mathbb{F}_q \ (i \leq n)$.
Therefore, $\mathbb{F}_q[t], \mathbb{K}$, and $\mathbb{K}_{\infty}$ play the roles of $\mathbb{Z}, \mathbb{Q}$, and $\mathbb{R}$,
respectively. Let
$$
\mathbb{T} = \left\{   \sum_{i= - \infty}^{-1} a_i t^i  : a_i \in \mathbb{F}_q \ (i \leq -1) \right\},
$$
which is the analogue of the unit interval $[0,1)$. 

For $\alpha = \sum_{i= - \infty}^{n} a_i t^i \in \mathbb{K}_{\infty}$, if $a_n \not = 0$, we define
$\ord \alpha = n.$ We say $\alpha$ is \textit{rational} if  $\alpha \in \mathbb{K}$ and
\textit{irrational} if  $\alpha \not \in \mathbb{K}$. We define $\{ \alpha \} = \sum_{i= - \infty}^{-1} a_i t^i
\in \mathbb{T}$ to be the \textit{fractional} part of $\alpha$.
We refer to $a_{-1}$ as the \textit{residue} of $\alpha$, denoted by $\res \alpha$.
We now define the exponential function on $\mathbb{K}_{\infty}$. Let
$\text{tr}: \mathbb{F}_q \rightarrow \mathbb{F}_p$ denote the familiar trace map.
There is a non-trivial additive character $e_q: \mathbb{F}_q \rightarrow \mathbb{C}^{\times}$
defined for each $a \in \mathbb{F}_q$ by taking $e_q(a) = e^{ 2 \pi i (\text{tr}(a)/p) }$.
This character induces a map $e : \mathbb{K}_{\infty} \rightarrow \mathbb{C}^{\times}$
by defining, for each element $\alpha \in \mathbb{K}_{\infty}$, the value of $e(\alpha)$
to be $e_q(\res \alpha)$. For $N \in \mathbb{Z}^+$, we write $\mathbb{G}_N$ for the
set of all polynomials in $\mathbb{F}_q[t]$ whose degree are less than $N$.

Given $j,r \in \mathbb{Z}^+$, we write $j \preceq_p r$ if $p \nmid {r \choose j}$.
By Lucas' Theorem, this happens precisely when all the digits of $j$ in base $p$
are less than or equal to the corresponding digits of $r$. From this characterization,
it is easy to see that the relation $\preceq_p$ defines a partial order on $\mathbb{Z}^+$.
If $j \preceq_p r$, then we necessarily have $j \leq r$. Let $\mathcal{K} \subseteq \mathbb{Z}^+$.
We say an element $k \in \mathcal{K}$  is \textit{maximal} if it is maximal with respect to
$\preceq_p$, that is, for any $r \in \mathcal{K}$, either $r \preceq_p k$
or $r$ and $k$ are not comparable. Following the notation of \cite{LL}, we define
the \textit{shadow} of $\mathcal{K}$, $\mathcal{S(K)}$, to be
$$
\mathcal{S(K)} = \left\{  j \in \mathbb{Z}^+ : 
j \preceq_p r \text{ for some } r \in \mathcal{K}  \right\}.
$$
We also define
$$
\mathcal{K}^* = \left\{ k \in \mathcal{K} : p \nmid k \text{ and }
p^vk \not \in \mathcal{S(K)} \text{ for any } v \in \mathbb{Z}^+   \right\}.
$$

Given $f(u) \in \mathbb{K}_{\infty}[u]$, we mean by $f(u)$ is \textit{supported on a set} $\mathcal{K} \subseteq \mathbb{Z}^+$
that $f(u) = \sum_{r \in \mathcal{K} \cup \{ 0 \}} \alpha_r u^r$, where $0 \not = \alpha_r \in \mathbb{K}_{\infty} \ (r \in \mathcal{K} )$.
As explained in the remark of \cite[Theorem 12]{LL}, the non-zero coefficient $\alpha_k$, for $k \in \mathcal{K}^*$ which is maximal in $\mathcal{K}$, plays the role of the leading coefficient of the polynomial. This is, in a sense,
the  \textquotedblleft  true"
$\mathbb{F}_q[t]$ analogue of the leading coefficient.

\begin{example}
\label{example one}
Let $p > 3$. The polynomial $f_1 (u) = c_{2p^2+p} u^{2p^2 + p}+ c_{3p+1} u^{3p+1} + c_p u^p + c_2 u^2 + c_1 u^1 + c_0$,
where each $c_j \not = 0$, 
is supported on $\mathcal{K}_1 = \{2p^2+p, 3p+1, p, 2, 1 \}$.
We can verify that
$$
\mathcal{S} (\mathcal{K}_1) = \left\{  2p^2 + p, 2p^2, p^2 + p, p^2,  p, 3p+1, 2p+1, p+1, 3p, 2p, 2,1 \right\}
$$
and
$$
\mathcal{K}_1^* = \left\{ 3p + 1 \right\}.
$$
\end{example}

We are now in position to state one of our main results.
The following theorem is an analogue of Theorem \ref{thm LS 1}.
\begin{thm}
\label{first theorem}
Let $h(u) = \sum_{r \in \mathcal{K} \cup \{ 0 \}} c_r u^r$ be a polynomial supported on a set $\mathcal{K} \subseteq \mathbb{Z}^+$
with coefficients in $\mathbb{F}_{q}[t]$. Suppose $c_k \not = 0$ for some $k \in \mathcal{K}^*$.
Suppose further that for every $g$ in $\mathbb{F}_q[t] \backslash \{ 0\}$, there exists an $m_g \in \mathbb{G}_{\deg g}$ such that $h(m_g) \equiv 0 \ (\text{mod }g).$ Then there exist $\theta = \theta(\mathcal{K}, q, \deg h )  > 0$ 
and $N_0 = N_0( \mathcal{K}, q, h, \theta )  \in \mathbb{Z}^+$ 
such that for any $N > N_0$, we have
$$
\min_{x \in \mathbb{G}_N} \ord \{ \beta h(x) \} \leq - \theta N
$$
uniformly in $\beta \in \mathbb{K}_{\infty}$.
\end{thm}
A set $\mathcal{H} \subseteq \mathbb{F}_q[t] \backslash \{ 0\}$ is said to be \textit{van der Corput}
if the sequence $(a_x)_{x \in \mathbb{F}_q[t]} \subseteq \mathbb{K}_{\infty}$ is equidistributed in $\mathbb{T}$ (defined analogously
as in the case of $\mathbb{R}$), whenever the sequence $(a_{x+h} - a_x)_{x \in \mathbb{F}_q[t]}$ is equidistributed in $\mathbb{T}$ for each $h \in \mathcal{H}.$
We remark that given a polynomial $h(u)$ that satisfies the hypothesis of Theorem \ref{first theorem},
the set
$\{ h(x) : x \in \mathbb{F}_q[t] \} \backslash \{ 0 \}$ is van der Corput \cite[Theorem 23]{LL},
a topic which we do not get into in our current chapter. We instead refer
the reader to \cite{Le} and \cite{LL} for more information on this topic.

It is clear that any polynomial $h(u) \in \mathbb{F}_{q}[t][u]$ such that $(u - a) | h(u)$ for some $a \in \mathbb{F}_q[t]$
satisfies the hypothesis of Theorem \ref{first theorem}. However, polynomials in $\mathbb{F}_{q}[t][u]$ with roots in $\mathbb{F}_q[t]$
are not the only elements satisfying this condition.
\begin{example}
\label{example two}
Let $p=5$ and consider $h(u) = (u^2 - t)(u^2 - (t+1)) (u^2 - (t^2+t)) \in \mathbb{F}_{5}[t][u]$. Then $h(u)$ does not have a root in $\mathbb{F}_5[t]$, but it has a root modulo $g$ for every $g$ in $\mathbb{F}_5[t] \backslash \{ 0\}$.
\end{example}
We postpone the proof of this statement to \ref{A}.
L\^{e} and Spencer also proved the following theorem in \cite{LS}.
\begin{thm} \textnormal{ \cite[Theorem 6]{LS} }
\label{thm 6 LL}
Suppose the polynomials $h_1, ..., h_L$ of distinct degrees are such that any linear combination
of them with integer coefficients has a root modulo $q$ for any $q \in \mathbb{N}$.
Let $\alpha_1, ..., \alpha_L \in \mathbb{R}$. Then there is an exponent $\theta > 0$ (depending
at most on $h_1, ..., h_L$) such that
$$
\min_{1 \leq n \leq N} \| \alpha_1 h_1(n)  + ... + \alpha_L h_L(n)   \| \ll N^{- \theta}
$$
uniformly in $\alpha_1, ..., \alpha_L, N$.
\end{thm}

Suppose we have polynomials $h_1, ..., h_L \in \mathbb{F}_{q}[t] [u]$, where
$h_j(u) = \sum_{r \in \mathcal{K}_j \cup \{ 0 \}} c_{j,r} u^r$,
and $\mathcal{K}_j \subseteq \mathbb{Z}^+ \ (1 \leq j \leq L)$.
Let $\mathcal{K} = \mathcal{K}_1 \cup ... \cup \mathcal{K}_L$.
We define the  \textit{$\mathcal{K}^*$-portion} of $h_j$   as
$$
h_j^ 
* (u):= \sum_{r \in \mathcal{K}_j \cap \mathcal{K}^*} c_{j,r} u^r.
$$
We say the \textit{$\mathcal{K}^*$-portion of $( h_j )_{j=1}^L$ is linearly independent}
if $h_1^{*}, ..., h_L^{*}$ are linearly independent over $\mathbb{K}$.
We also define a slightly stronger notion, the \textit{maximal $\mathcal{K}^*$-portion} of $h_j$ as
$$
h_j^{\text{max }} (u):= \sum_{
\substack{
r \in \mathcal{K}_j \cap \mathcal{K}^* \\
r \text{ is maximal in } \mathcal{K}
} } c_{j,r} u^r.
$$
We say the \textit{maximal $\mathcal{K}^*$-portion of $( h_j )_{j=1}^L$ is linearly independent}
if $ h_1^{\text{max }}, ... , h_L^{\text{max }}$ are linearly independent over $\mathbb{K}$.
Clearly, if the maximal $\mathcal{K}^*$-portion of $( h_j )_{j=1}^L$ is linearly independent, then
the $\mathcal{K}^*$-portion of $( h_j )_{j=1}^L$ is linearly independent.
We give an example of these notions below.
\begin{example}
Let $p>3$. Consider polynomials $f_1 (u) = c_{2p^2+p} u^{2p^2 + p}+ c_{3p+1} u^{3p+1} + c_p u^p + c_2 u^2 + c_1 u^1 + c_0$ and
$f_2 (u) = c'_{p^3+3p+1} u^{p^3+3p + 1} + c'_{p^2+1} u^{p^2 + 1} + c'_{2p+1} u^{2p+1}$ in $\mathbb{F}_q[t][u]$, where each $c_j$ and $c'_{j'}$
are non-zero elements of $\mathbb{F}_q[t]$. In other words, $f_1 (u)$ and $f_2 (u)$ 
are supported on $\mathcal{K}_1 = \{2p^2+p, 3p+1, p, 2, 1 \}$ and $\mathcal{K}_2 = \{p^3+3p+1, p^2+1, 2p+1 \}$, respectively.
Thus we let
$$
\mathcal{K} = \mathcal{K}_1 \cup \mathcal{K}_2 = \{p^3+ 3p + 1, 2p^2+p, p^2+1, 3p+1, 2p+1, p, 2, 1 \},
$$
and we can verify that
$$
\mathcal{K}^* = \{p^3 + 3p + 1, 3p+1 \}.
$$
It follows that
$$
f_1^*(u) = c_{3p+1} u^{3p+1} \ \ \text{ and } \ \ f_2^*(u) = c'_{p^3+3p+1} u^{p^3+3p+1},
$$
which are clearly linearly independent over $\mathbb{K}$. Therefore, $\mathcal{K}^*$-portion of $( f_1, f_2)$ is linearly independent.
However, note $p^3 + 3p + 1$ is maximal in $\mathcal{K}$, but not $3p+1$. Thus we have
$$
f_1^{\text{max }}(u) = 0 \ \ \text{ and } \ \ f_2^{\text{max }}(u) = c'_{p^3+3p+1} u^{p^3+3p+1},
$$
and consequently, the maximal $\mathcal{K}^*$-portion of $( f_1, f_2)$ is not linearly independent.
\end{example}

The following theorem is an analogue of Theorem \ref{thm 6 LL}.
\begin{thm}
\label{second theorem}
Let $h_j \in \mathbb{F}_{q}[t] [u]$ be supported on a set $\mathcal{K}_j \subseteq \mathbb{Z}^+ \ (1 \leq j \leq L)$, and let
$\mathcal{K} = \mathcal{K}_1 \cup ... \cup \mathcal{K}_L$.
Suppose
any linear combination of them with $\mathbb{F}_q[t]$ coefficients has a root modulo $g$ for any $g \in \mathbb{F}_q[t] \backslash \{ 0\}$.
Suppose further that the $\mathcal{K}^*$-portion of $( h_j )_{j=1}^L$ is linearly independent.
Then there exist $\theta = \theta \left( \mathcal{K}, q, \max_{ 1\leq j \leq L} \deg h_j \right) > 0$
and $N_0 = N_0(\mathcal{K}, q, \theta, h_1, ..., h_L) \in \mathbb{Z}^+$ such that for any $N > N_0$, we have
$$
\min_{x \in \mathbb{G}_N} \ord \{ \beta_1 h_1(x) + ... + \beta_L h_L(x) \} \leq -\theta N
$$
uniformly in $\beta_1, ..., \beta_L \in \mathbb{K}_{\infty}$.
\end{thm}
We give an example of a system of polynomials $(h_1, h_2) \subseteq \mathbb{F}_5[t][u]$ that satisfies the hypothesis of Theorem \ref{second theorem}
in Example \ref{example three}. We note that these polynomials $h_1(u)$ and $h_2(u)$ do not have a common root in $\mathbb{F}_5[t]$, but they do have $h(u)$ from Example \ref{example two} as a common factor. There may well be examples of systems $(h_1, ..., h_L)$ without a common factor
such that any linear combination of them with $\mathbb{F}_q[t]$ coefficients has a root modulo $g$ for any $g \in \mathbb{F}_q[t] \backslash \{ 0\}$,
but we do not have such an example in hand at this time.



We also prove an analogue of \cite[Theorem 7]{LS} in Theorem \ref{third theorem},
which is a (partial) generalization of Theorem \ref{second theorem}.
However, we defer stating the result to Section \ref{proving} in order to avoid introducing further
notation here.

The organization of the rest of the paper is as follows. In Section \ref{prelim}, we
introduce some notation and notions required to carry out our discussions in the setting over
$\mathbb{F}_q[t]$. In Section \ref{basic lin alg}, we prove lemmas involving basic linear
algebra utilized in the proof of our main results given in Section \ref{proving}. We provide
the proof of the statement in Example \ref{example two} in \ref{A}.
We note that L\^{e} and Spencer generalized \cite[Theorem 7]{LS}, which Theorem \ref{third theorem}
is an analogue of, and obtained results on simultaneous approximation \cite[Theorems 4 and 8]{LS}.
However, due to complications that arose during our attempt from certain arguments in linear algebra and geometry of numbers
in the setting over $\mathbb{F}_q[t]$, at present time we decided to leave generalizing Theorem \ref{third theorem}
in a similar manner as a possible future work. Finally, in the case when the polynomials
in question do not have constant terms, a more general result is available due to Spencer and Wooley \cite{SW}.
We note that their result does not require the extra hypothesis on the coefficients as in this paper.


\section{Preliminaries}
\label{prelim}

Suppose a system of polynomials $(h_1, ..., h_L)$ satisfies the following,
\newline
\newline
Condition $ ( \star ) $: For every $g \in \mathbb{F}_q[t] \backslash \{ 0\}$,
there exists $m_g \in \mathbb{F}_q[t]$ such that $h_i(m_g) \equiv 0 \ (\text{mod }g)$
for $i = 1, ..., L$.
\newline
\newline
In the case of $\mathbb{Z}$ (in place of $\mathbb{F}_q[t]$), such a system of polynomials
satisfying the analogous condition is called \textit{jointly interesective polynomials}.

We have the following analogue of \cite[Proposition 6.1]{BAL}, which we omit the proof of. 
\begin{lem}
\label{joint to single}
A system of polynomials $(h_1, ..., h_L)$ in $\mathbb{F}_q[t][u]$ satisfies Condition $ ( \star ) $
if and only if there exists a polynomial $d \in \mathbb{F}_q[t][u]$, which has a root
modulo $g$ 
for every $g \in \mathbb{F}_q[t] \backslash \{ 0\}$, and $d | h_i \ (1 \leq i \leq L)$
over $\mathbb{F}_q[t]$.
\end{lem}

Let $w$ be a monic irreducible polynomial in $\mathbb{F}_q[t]$.
Let $\lambda_N$ be the canonical projection
from $\mathbb{F}_q[t] / w^{N+1} \mathbb{F}_q[t]$ to $\mathbb{F}_q[t] / w^N \mathbb{F}_q[t]$.
For each $w$, we define the projective limit
$$
\lim_{ \substack{\leftarrow \\ N} } \mathbb{F}_q[t] / w^N \mathbb{F}_q[t]
=
\Big{\{} (x_i)_{i \in \mathbb{N}} \in \prod_{i=1}^{\infty} \mathbb{F}_q[t] / w^i \mathbb{F}_q[t] : \lambda_i(x_{i+1}) = x_i, i = 1, 2, ... \Big{\}}.
$$
Take $\bar{x} = (x_i)_{i \in \mathbb{N}}  \in  \lim\limits_{ \substack{\leftarrow  \\ N }} \mathbb{F}_q[t] / w^N \mathbb{F}_q[t].
$
We say that $\bar{x}$ is a solution to the equation $f(u) = 0$, if $\bar{x}$ satisfies
$$
f(x_i) \equiv 0 \ (\text{mod } w^i)
$$
for all $i \in \mathbb{N}$.

We have the following lemma, which its proof follows closely that of the $p$-adic integers, for
example see \cite[Chapter II, Proposition 1.4]{N}. 
\begin{lem}
\label{p adic lemma 2}
Let $f$ be a polynomial in $\mathbb{F}_q[t][u]$ and $w$ a monic irreducible in $\mathbb{F}_q[t]$. Then $f$ has a root modulo $w^N$ 
for every
$N \in \mathbb{N}$ if and only if the equation $f(u) = 0$ has a solution in $\lim\limits_{ \substack {\leftarrow \\ N} } \mathbb{F}_q[t] / w^N \mathbb{F}_q[t]$.
\end{lem}
We leave the following lemma as an exercise for the reader.
\begin{lem}
\label{p adic lemma 3}
Let $f$ be a polynomial in $\mathbb{F}_q[t][u]$. Then $f$ has a root modulo $g$ 
for every
$g \in \mathbb{F}_q[t] \backslash \{ 0\}$ if and only if for every monic irreducible $w$,
the equation $f(u) = 0$ has a solution in $\lim\limits_{ \substack{ \leftarrow \\ N} } \mathbb{F}_q[t] / w^N \mathbb{F}_q[t]$.
\end{lem}

Corresponding to any system of polynomials $(h_1, ..., h_L)$ satisfying Condition $ ( \star ) $,
there exists $d \in \mathbb{F}_q[t][u]$ satisfying the conditions of Lemma \ref{joint to single}.
Given a monic irreducible $w$, by Lemma \ref{p adic lemma 3}, we know there exists
$(r_{w^j}) \in \lim\limits_{ \substack{ \leftarrow \\ N} } \mathbb{F}_q[t] / w^N \mathbb{F}_q[t]$
which is a solution to $d(u)=0$, in other words $d(r_{w^j}) \equiv 0 \ (\text{mod } w^j)$
and $r_{w^j} \equiv r_{w^{j+1}} \ (\text{mod }w^j)$ for all $j \in \mathbb{N}$. We fix such a solution for each $w$.
Suppose we are given $g = a \prod_{i=1}^T w_i^{S_i} = a g_1$, where the $w_i$'s are distinct monic
irreducibles in $\mathbb{F}_q[t]$ and $a \in \mathbb{F}_q$. By the Chinese Remainder Theorem, we define $r_g$ to be the unique
element in $\mathbb{F}_q[t] / (g_1)$ such that $r_g \equiv r_{w_i^{S_i}} \ ( \text{mod } w_i^{S_i} )$ for $1 \leq i \leq T$.
Since $d(r_g) \equiv 0 \ (\text{mod } w_i^{S_i} )$ for $1 \leq i \leq T$, it follows that
$d(r_g) \equiv 0 \ (\text{mod } g)$.
Suppose we have $y = b \prod_{i=1}^T w_i^{S'_i}$, where $S'_i \leq S_i$ and $b \in \mathbb{F}_q$, so that $y | g$. Then since
$r_g \equiv r_{w_i^{S_i}} \equiv r_{w_i^{S'_i}} \ (\text{mod } w_i^{S'_i})$ for $1 \leq i \leq T$,
we obtain $r_g \equiv r_y (\text{mod } y)$. 
Finally, for $a \in \mathbb{F}_q$ we let $r_a = 0$.

Therefore, corresponding to any system of polynomials $(h_1, ..., h_L)$ satisfying Condition $ ( \star ) $,
we can associate a sequence $(r_x)_{ x \in \mathbb{F}_q[t] \backslash \{ 0 \} } \subseteq \mathbb{F}_q[t]$
such that for any $m, y \in \mathbb{F}_q[t] \backslash \{ 0 \}$, $r_y \in \mathbb{G}_{\ord y}, \ r_{my} \equiv r_y \ (\text{mod } y)$,
and
\begin{equation}
\label{r_d's}
h_j(r_y) \equiv 0 \ (\text{mod } y) \ \ (1 \leq j \leq L).
\end{equation}
We note that the approach to define the sequence $(r_x)_{ x \in \mathbb{F}_q[t] \backslash \{ 0 \} }$
here was taken from 
\cite{Lu}, which deals with the case of $\mathbb{Z}$.

For any element $\alpha \in \mathbb{K}_{\infty}$, it is easy to see that
$$
\ord \{ \alpha \} = \min_{z \in \mathbb{F}_q[t]} \ord (\alpha - z),
$$
where the minimum is achieved when $z = \alpha - \{ \alpha \}$, the \textit{integral} part of $\alpha$.
Also for $\alpha_1, ..., \alpha_L \in \mathbb{K}_{\infty}$, we have
\begin{eqnarray}
\notag
\ord \Big{\{} \sum_{j=1}^L \alpha_j  \Big{ \} }
&\leq& \ord \left( \sum_{j=1}^L \alpha_j - \sum_{j=1}^L (\alpha_j - \{ \alpha_j \}) \right)
\\
&=& \ord \left( \sum_{j=1}^L \{ \alpha_j \} \right)
\notag
\\
&\leq& \max_{1 \leq j \leq L} \ord \{ \alpha_j \}.
\label{ord fractional part}
\end{eqnarray}

\begin{lem}
\label{first lemma}
Let $\beta_1, \beta_2, ..., \beta_R \in \mathbb{K}_{\infty}$ and suppose $\ord \{ \beta_j\} \geq -M \ (1 \leq j \leq R)$.
Then there exists $x \in \mathbb{G}_M \backslash \{ 0 \}$ such that
$$
\Big{|} \sum_{j=1}^R  e(x \beta_j) \Big{|} \geq \frac{R}{q^M-1}.
$$
\end{lem}

\begin{proof}
For $\alpha \in \mathbb{K}_{\infty}$, we have by \cite[Lemma 7]{RMK} 
\begin{eqnarray}
\label{kubota}
\sum_{x \in \mathbb{G}_M} e(x \alpha)=
\left\{
    \begin{array}{ll}
         q^M, &\mbox{if } \ord \{\alpha \} <  -M, \\
         0, &\mbox{if  } \ord \{\alpha \} \geq  -M. \\
    \end{array}
\right.
\end{eqnarray}
Since $\ord \{ \beta_j \} \geq -M \ (1 \leq j \leq R)$, we have
$$
\sum_{j=1}^R \sum_{x \in \mathbb{G}_M} e(x \beta_j) = 0.
$$
Therefore, it follows that
$$
\sum_{x \in \mathbb{G}_M \backslash \{ 0 \}} \Big{|} \sum_{j=1}^R  e(x \beta_j) \Big{|} \geq R,
$$
from which we obtain our result.
\end{proof}

We invoke the following result from \cite{LL}. The theorem allows us to estimate
certain coefficients of a polynomial $f(u)$ by an element in $\mathbb{K}$ when the exponential sum
of $f(u)$ is sufficiently large.
\begin{thm} \textnormal{\cite[Theorem 15]{LL}} 
\label{thm LL}
Let $f(u) = \sum_{r \in \mathcal{K} \cup \{ 0 \}} \alpha_r u^r$ be a polynomial supported on a set $\mathcal{K} \subseteq \mathbb{Z}^+$
with coefficients in $\mathbb{K}_{\infty}$. Then for any $k \in \mathcal{K}^*$, there exist
constants $c_k, C_k > 0$, depending only on $\mathcal{K}$ and $q$, such that the following holds:
suppose that for some $0 < \eta \leq c_k N$, we have
$$
\Big{|} \sum_{x \in \mathbb{G}_N}  e(f(x)) \Big{|} \geq q^{N- \eta}.
$$
Then for any $\epsilon > 0$ and $N$ sufficiently large in terms of $\mathcal{K}$, $\epsilon$ and $q$,
there exist $a_k, g_k \in \mathbb{F}_q[t]$ such that
$$
\ord (g_k \alpha_k - a_k) < -k N + \epsilon N + C_k \eta \ \ \text{  and  } \ \ \ord g_k \leq \epsilon N + C_k \eta.
$$
\end{thm}
We have the following corollary where we replace the polynomial $g_k \in \mathbb{F}_q[t]$ and constants $c_k,C_k >0$
in the statement of Theorem \ref{thm LL} with $g \in \mathbb{F}_q[t]$ and $c,C >0$, which are independent of
the choice of $k \in \mathcal{K}^*$, respectively.

\begin{cor}
\label{Cor LL main}
Let $f(u) = \sum_{r \in \mathcal{K} \cup \{ 0 \}} \alpha_r u^r$ be a polynomial supported on a set $\mathcal{K} \subseteq \mathbb{Z}^+$
with coefficients in $\mathbb{K}_{\infty}$. There exist
constants $c, C > 0$, depending only on $\mathcal{K}$ and $q$, such that the following holds:
suppose that for some $0 < \eta \leq c N$, we have
$$
\Big{|} \sum_{x \in \mathbb{G}_N}  e(f(x)) \Big{|} \geq q^{N- \eta}.
$$
Then for any $\epsilon > 0$ and $N$ sufficiently large in terms of $\mathcal{K}$, $\epsilon$ and $q$,
there exists $g \in \mathbb{F}_q[t]$ such that
$$
\ord \{ g \alpha_k \} < -k N + \epsilon N + C \eta \ \ ( k \in \mathcal{K}^*) \text{  and  } \ \ \ord g \leq \epsilon N + C \eta.
$$
\end{cor}
\begin{proof}
For each $k \in \mathcal{K}^*$, let $c_k, C_k$ be the constants, depending only on $\mathcal{K}$ and $q$, and $a_k$, $g_k$ be the polynomials from the statement
of Theorem \ref{thm LL}. Let $c = \min_{k \in \mathcal{K}^*} c_k$ and $C = \max_{k \in \mathcal{K}^*} C_k$.
We let $g = \prod_{k \in \mathcal{K}^*} g_k$ and $C' = | \mathcal{K}^* | C$.
Since $\ord g_k \leq \epsilon N + C_k \eta \ (k \in \mathcal{K}^*)$, it follows that
$$
\ord g \leq 
| \mathcal{K}^* | \epsilon N + C' \eta.
$$
We also obtain
$$
\ord \{ g \alpha_k \} \leq
\ord \left( g \alpha_k - a_k \prod_{j \in \mathcal{K}^* \backslash \{ k \}} g_j \right)
\leq - k N + | \mathcal{K}^* | \epsilon N + C' \eta.
$$
\end{proof}

We note that all of our main results, Theorems \ref{first theorem}, \ref{second theorem} and \ref{third theorem},
rely on Corollary \ref{Cor LL main}, which explains the reason for our assumptions on the coefficients of the polynomials
in these theorems.

\section{Basic Linear Algebra}
\label{basic lin alg}

In this section, we prove lemmas involving basic linear algebra which are utilized in the proofs of our main results.
Given a polynomial $f(u) \in \mathbb{K}_{\infty}[u]$, we use the notation
$[f]_{i}$ to mean the $u^i$ coefficient of $f$.
We have the following lemma, which is an analogue of \cite[Lemma 1]{LS}.
\begin{lem}
\label{matrix lemma}
Suppose $d,s \in \mathbb{F}_q[t]$, $d \not = 0$, and $f_1, ..., f_L \in \mathbb{F}_q[t][u]$
with $\deg f_1 < ... < \deg f_L$. There exist polynomials $g_1, ..., g_L \in \mathbb{F}_q[t][u]$,
depending on $d$ and $s$,
and an $L \times L$ matrix $\mathcal{A}$ with entries in $\mathbb{F}_q[t]$ satisfying the following properties:
\newline
\newline
(1) $\mathcal{A}
\left(\begin{array}{c}
f_1(d u + s) \\ \vdots \\ f_L(d u + s)\\
\end{array} \right)
=
\left(\begin{array}{c}
g_1(u ) \\ \vdots \\ g_L(u)\\
\end{array} \right)$
\newline
\newline
(2) $\mathcal{A}$ is lower triangular with entries in $\mathbb{F}_q[t]$. All its diagonal entries
are equal to a constant $c \in \mathbb{F}_q[t]$ depending only on $f_1, .., f_L$.
In fact, every entry of $\mathcal{A}$ is dependent at most on $s$ and $f_1, .., f_L$.
\newline
\newline
(3) We have $[g_i]_{\deg g_j} = 0$ if $i \not = j$. Also,
$\deg g_j = \deg f_j$ and $[g_j]_{\deg g_j} = c d^{\deg f_j} [f_j]_{deg f_j}$ for all $1 \leq j \leq L$.
\end{lem}
\begin{proof}
Let $\mathcal{A}' = (a_{i,j})$ be a lower triangular matrix with all entries on the main
diagonal equal to $1$. For each $1 \leq i \leq L$, one can successively select elements
in $\mathbb{K}$, $a_{i,i-1}, ... , a_{i,1}$ so that in the polynomial
$$
h_{i}(u) = a_{i,1} f_1(du + s) + a_{i,2} f_2(du + s)  + ... + a_{i,i-1} f_{i-1}(du + s) + f_i(du + s),
$$
the coefficient of $u^{\deg f_j}$ is $0$ for every $j < i$.
We prove by induction that $a_{i,j} \ (j < i)$ depend only on $s$ and $f_1, ..., f_L$, and
that their denominators  
depend only on $f_1, ..., f_L$.
Fix $1 \leq i \leq L$. For the base case $j = i-1$, we have
\begin{eqnarray}
0 &=& [h_i]_{\deg f_{i-1}}
\notag
\\
&=&  [ a_{i,i-1} f_{i-1}(du + s) + f_i(du+s) ]_{\deg f_{i-1}}
\notag
\\
&=&  a_{i,i-1} [f_{i-1}]_{\deg f_{i-1}} d^{\deg f_{i-1}} +
\sum_{l = \deg f_{i-1}}^{\deg f_{i}} [f_{i}]_l \ {l \choose \deg f_{i-1}} \ d^{\deg f_{i-1}} s^{l - \deg f_{i-1}}.
\notag
\end{eqnarray}
By rearraging the last equality above, we obtain the following equaiton
$$
a_{i,i-1} = \frac{-1}{[f_{i-1}]_{\deg f_{i-1}}} \sum_{l = \deg f_{i-1}}^{\deg f_{i}}
[f_{i}]_l \ {l \choose \deg f_{i-1}} \ s^{l - \deg f_{i-1}},
$$
which we deduce our base case from.
Suppose the statement holds for $j_0 < j < i$.
Then we have by similar calculations as above and the induction hypothesis that
\begin{eqnarray}
0 &=& [h_i]_{\deg f_{j_0}}
\notag
\\
&=&  [ a_{i,j_0} f_{j_0}(du + s) + ... + a_{i,i-1} f_{i-1}(du + s) + f_i(du+s) ]_{\deg f_{j_0}}
\notag
\\
&=& 
a_{i,j_0} d^{\deg f_{j_0}} [f_{j_0}]_{\deg f_{j_0}} + d^{\deg f_{j_0}} \ \widetilde{a},
\notag
\end{eqnarray}
where $\widetilde{a} \in \mathbb{K}$ depends only on $s$ and $f_1, ..., f_L$,
and its denominator depends only on $f_1, ..., f_L$.
We then obtain our claim for $j = j_0$ by rearranging the last equation displayed above.
Let $c \in \mathbb{F}_q[t]$ be the common denominator of the non-zero entries in $\mathcal{A}'$; the matrix $\mathcal{A} = c \mathcal{A}'$
and the polynomials $g_j (u)= c h_j(u) \ (1 \leq j \leq L)$ satisfy the desired properties.
\end{proof}

By Lemma \ref{matrix lemma}, we obtain Lemmas \ref{cor to matrix lemma}
and \ref{second cor to matrix lemma} which involve polynomials with
$\mathcal{K}^*$-portion and maximal $\mathcal{K}^*$-portion, respectively,
that are linearly independent.

\begin{lem}
\label{cor to matrix lemma}
Let $h_j \in \mathbb{F}_{q}[t] [u]$ be supported on a set $\mathcal{K}_j \subseteq \mathbb{Z}^+ \ (1 \leq j \leq L)$,
and let $\mathcal{K} = \mathcal{K}_1 \cup ... \cup \mathcal{K}_L$.
Suppose the $\mathcal{K}^*$-portion of $( h_j )_{j=1}^L$ is linearly independent.
Let $\beta_1, ..., \beta_L \in \mathbb{K}_{\infty}$. 
Then we can find an $L \times L$ matrix $\mathcal{T}$ with entries in $\mathbb{F}_q[t]$ and
$g_j \in \mathbb{F}_{q}[t] [u] \ (1 \leq j \leq L)$ 
with the following properties:
\newline
\newline
(1) $g_j$ is a polynomial supported on a subset of $\mathcal{K}$.
\newline
\newline
(2) $\mathcal{T}  
\left(\begin{array}{c}
h_1(u ) \\ \vdots \\ h_L(u)\\
\end{array} \right)
=
\left(\begin{array}{c}
g_1(u ) \\ \vdots \\ g_L(u)\\
\end{array} \right)$.
\newline
\newline
(3) There exist $T_j \in \mathcal{K}^* \ (1 \leq j \leq L)$ such that
$[g_i]_{T_j} = 0$ if $i \not = j$.
\newline
\newline
(4) There exist $\gamma_j \in \mathbb{K}_{\infty} \ (1 \leq j \leq L)$ such that
$$
\beta_1 h_1(u) + ... + \beta_L h_L(u) = \gamma_1 g_1(u) + ... + \gamma_L g_L(u).
$$
\end{lem}

\begin{proof}
By the hypothesis, the polynomials $\{ h_j^{*} \}_{j=1}^L$
are linearly independent over $\mathbb{K}$. Therefore, we can find an $L \times L$ invertible matrix $\mathcal{B}$
with entries in $\mathbb{K}$ such that
$$
\mathcal{B} \ (h_1^{*}, ..., h_L^{*} )^T =
(b_1, ..., b_L)^T,
$$
where $b_j \in \mathbb{F}_q[t][u]$ with coefficients
supported on a subset of $\mathcal{K}^*$ and $\deg b_1 < ... < \deg b_L$. 
Let $\mathcal{A}$ and $g'_1, ..., g'_L$ be the matrix and polynomials, respectively, obtained by applying
Lemma \ref{matrix lemma} to the polynomials $b_1, ..., b_L$ with $d=1$ and $s=0$.
It follows that the polynomials $g'_j$ have coefficients supported on a subset of $\mathcal{K}^*$.
Let $T_j = \deg g'_j = \deg b_j \in \mathcal{K}^* (1 \leq j \leq L)$.
Also let
$$
(g''_1, ..., g''_L)^T = \mathcal{A} \mathcal{B} \ (h_1 - h_1^{*}, ..., h_L - h_L^{*} )^T
$$
and
$$
g_j := g'_j + g''_j \ (1 \leq j \leq L).
$$
Let $c_j$ be the common denominator of the coefficients of $g_j \in \mathbb{K}[u] \ (1 \leq j \leq L)$,
$c'$ be the common denominator of the matrix $\mathcal{AB}$, and
$c = c' \prod_{j=1}^L c_j$.
By construction, we see that
$c g_j$ is a polynomial in $\mathbb{F}_q[t][u]$ with coefficients supported
on a subset of $\mathcal{K}$,
$$
(c g_j)^{*} = c (g_j^{*}) = c g'_j,
$$
and
$$
( c g_1(u), ... , c g_L(u) )^T = c \mathcal{AB} \ ( h_1(u), ... ,  h_L(u) )^T. 
$$
Since $[g'_i]_{T_j} = 0$ if $i \not = j$, it follows that $[c g_i]_{T_j} = 0$ if $i \not = j$.
Let
$$
(\gamma_1, ..., \gamma_L) = (\beta_1, ..., \beta_L) \ (c \mathcal{A} \mathcal{B})^{-1}.
$$
Then we have
$$
\gamma_1 c g_1(u) + ... + \gamma_L c g_L(u) = \beta_1 h_1(u) + ... + \beta_L h_L(u).
$$
Therefore, we see that the matrix $c \mathcal{A} \mathcal{B}$ and the polynomials $c g_j \ (1 \leq j \leq L)$
satisfy the desired properties.
\end{proof}

Let $f(u) = \sum_{r \in \mathcal{K} \cup \{ 0 \}} \alpha_r u^r$ be a polynomial supported on a set $\mathcal{K} \subseteq \mathbb{Z}^+$ with coefficients in $\mathbb{K}_{\infty}$. For any $r \in \mathcal{K}$ and $y, s \in \mathbb{F}_q[t]$, we have
$$
(y + s)^r = \sum_{j \preceq_p r} {r \choose j} y^j s^{r-j} + s^r.
$$
Therefore, for a fixed $s$, if $k$ is maximal in $\mathcal{K}$,
then there exist $\alpha'_j = \alpha'_j(\{ \alpha_r \}_{r \in \mathcal{K}};s) \in
\mathbb{K}_{\infty} \ (j \in \mathcal{S(K)} \backslash \{ k \})$ and
$\alpha'_0 = \alpha'_0(\{ \alpha_r \}_{r \in \mathcal{K} \cup \{ 0 \}} ;s) \in \mathbb{K}_{\infty}$
such that
$$
f(y + s) = \alpha_k (y+s)^k + \sum_{r \in \mathcal{K} \backslash \{ k \}} \alpha_r (y+s)^r + \alpha_0
= \alpha_k y^k + \sum_{j \in \mathcal{S(K)} \backslash \{ k \}} \alpha'_j y^j + \alpha'_0.
$$
In other words, the $y^k$ coefficient of $f(y)$ and $f(y+s)$ are the same.
Therefore, it follows that if $k_1, ..., k_M$ are maximal in $\mathcal{K}$, then
\begin{equation}
\label{maximal stuff}
f(y + s) = \sum_{i=1}^M \alpha_{k_i} y^{k_i} + \sum_{j \in \mathcal{S(K)} \backslash \{ k_1, ..., k_M \} } \alpha'_j y^j + \alpha'_0.
\end{equation}

\begin{lem}
\label{second cor to matrix lemma}
Let $h_j \in \mathbb{F}_{q}[t] [u]$ be supported on a set $\mathcal{K}_j \subseteq \mathbb{Z}^+ \ (1 \leq j \leq L)$, and
let $\mathcal{K} = \mathcal{K}_1 \cup ... \cup \mathcal{K}_L$.
Suppose the maximal $\mathcal{K}^*$-portion of $( h_j )_{j=1}^L$ is linearly independent.
Let $\beta_1, ...,  \beta_L \in \mathbb{K}_{\infty}$ and $s, d \in \mathbb{F}_q[t]$ with $d \not = 0$.
Then we can find
$g_j \in \mathbb{F}_{q}[t] [u] \ (1 \leq j \leq L)$, depending on $d$ and $s$, and
an $L \times L$ matrix $\mathcal{T}$ with entries in $\mathbb{F}_q[t]$
with the following properties:
\newline
\newline
(1) $g_j$ is a polynomial supported on a subset of $\mathcal{S}(\mathcal{K})$ and every entry of $\mathcal{T}$ depends
only on $h_1, ..., h_L$.
\newline
\newline
(2) $\mathcal{T}  
\left(\begin{array}{c}
h_1(du+s) \\ \vdots \\ h_L(du+s)\\
\end{array} \right)
=
\left(\begin{array}{c}
g_1(u ) \\ \vdots \\ g_L(u)\\
\end{array} \right)$.
\newline
\newline
(3) For $x \in \mathbb{F}_q[t]$, we have
$$
\ord g_j(x) \leq \left( \max_{1\leq j \leq L} \deg h_j \right) \ \ord (dx + s) + D,
$$
where $D$ is some constant dependent only on $h_1, ..., h_L$.
\newline
\newline
(4) There exist $T_j \in \mathcal{K}^* \ (1 \leq j \leq L)$
such that $T_j$ is maximal in $\mathcal{K}$ and $[g_i]_{T_j} = 0$ if $i \not = j$.
Moreover, we have $[g_j]_{T_j} = \tilde{c}_j d^{T_j}$ for some $\tilde{c}_j \in \mathbb{F}_q[t]$ dependent
only on $h_1, ..., h_L$.
\newline
\newline
(5) There exist $\gamma_j \in \mathbb{K}_{\infty} \ (1 \leq j \leq L)$ such that
$$
\beta_1 h_1(du+s) + ... + \beta_L h_L(du+s) = \gamma_1 g_1(u) + ... + \gamma_L g_L(u).
$$
\end{lem}

\begin{proof}
Let $h_j(u) = \sum_{r \in \mathcal{K}_j \cup \{ 0 \}} c_{j,r} u^r$ for $1 \leq j \leq L$.
We also let $\mathcal{H}_j = \{r \in \mathcal{K}_j \cap \mathcal{K}^*:
r \text{ is maximal in } \mathcal{K} \}$ so that
$
h_j^{\text{max}} (u) = \sum_{
r \in \mathcal{H}_j } c_{j,r} u^r.
$
Let $\mathcal{H} = \mathcal{H}_1 \cup ... \cup \mathcal{H}_L$.
We have by ~(\ref{maximal stuff}), the maximality condition of $r \in \mathcal{H}_j$, that
\begin{eqnarray}
\notag
h_j(d u + \widetilde{s}) = \sum_{r \in \mathcal{H}_j } c_{j,r} (du)^r
+ \sum_{ v \in \mathcal{S}(\mathcal{K}_j) \backslash \mathcal{H}_j } c'_{j,v} u^v + c'_{j,0}
\end{eqnarray}
for some $c'_{j,0}, c'_{j,v} \in \mathbb{F}_q[t] \ (1 \leq j \leq L, v \in \mathcal{S}(\mathcal{K}_j) \backslash \mathcal{H}_j).$
For any $l \not = j$, we have
$(\mathcal{S}(\mathcal{K}_j) \backslash \mathcal{H}_j )\cap \mathcal{H}_l = \varnothing$.
Suppose $v \in (\mathcal{S}(\mathcal{K}_j) \backslash \mathcal{H}_j) \cap \mathcal{H}_l$.
Since $v$ is maximal in $\mathcal{K}$, the only way $v$ can
be an element of $\mathcal{S}(\mathcal{K}_j)$ is if $v \in \mathcal{K}_j$. However, this forces
$v \in \mathcal{H}_j$ which is a contradiction.
Therefore, we can in fact write $h_j(d u + \widetilde{s})$ as
\begin{eqnarray}
\label{h_j(du+s)}
h_j(d u + \widetilde{s}) = h_j^{\text{max}} (d u)
+ \sum_{ v \in \mathcal{S}(\mathcal{K}) \backslash \mathcal{H} } c'_{j,v} u^v + c'_{j,0}.
\end{eqnarray}

By the hypothesis, the polynomials $\{ h^{ \text{max} }_j \}_{j=1}^L$ are linearly independent over $\mathbb{K}$.
Therefore, we can find an $L \times L$ invertible matrix $\mathcal{B}$
with entries in $\mathbb{F}_q[t]$ such that
$$
\mathcal{B} \ ({h}^{\text{max}}_1, ..., {h}^{\text{max}}_L)^T =
(b_1, ..., b_L)^T,
$$
where $b_j \in \mathbb{F}_q[t][u] \ (1 \leq j \leq L)$ with coefficients
supported on a subset of $\mathcal{H}$ and $\deg b_1 < ... < \deg b_L$. 
The entries of the matrix $\mathcal{B}$ and the polynomials $b_1, ..., b_L$ are dependent only on
${h}^{\text{max}}_1, ..., {h}^{\text{max}}_L.$
Clearly we have
$$
\mathcal{B} \ ({h}^{\text{max}}_1(du), ..., {h}^{\text{max}}_L(du))^T =
(b_1(du), ..., b_L(du))^T.
$$
Let $\mathcal{A}$ and $g'_1, ..., g'_L$ be the matrix and polynomials, respectively, obtained by applying
Lemma \ref{matrix lemma} to the polynomials $b_1, ..., b_L$ with 
$s=0$ and $d$. It follows that the coefficients of $g'_j \ (1 \leq j \leq L)$ are supported on a subset of $\mathcal{H}$.
Note by $(2)$ of Lemma \ref{matrix lemma}, the entries of $\mathcal{A}$ depend only
on $h_1, ..., h_L$.
Let $T_j = \deg g'_j = \deg b_j \in \mathcal{H} \ (1 \leq j \leq L)$. We have by $(3)$ of Lemma \ref{matrix lemma}
that $[g'_j]_{T_j} = cd^{T_j}[b_j]_{T_j}$ for some $c \in \mathbb{F}_q[t]$ dependent only
on $b_1, ..., b_L$, and $[g'_i]_{T_j} = 0$ if $i \not = j$.
Let
$$
(g''_1, ..., g''_L)^T = \mathcal{A} \mathcal{B} \ (h_1 (du + \widetilde{s}) - {h}^{\text{max}}_1(du), ..., h_L(du + \widetilde{s}) - {h}^{\text{max}}_L(du) )^T.
$$
We define polynomials $g_j$ by
$$
g_j := g'_j + g''_j \ (1 \leq j \leq L),
$$
then we have
\begin{equation}
\label{g to h}
(  g_1(u), ... ,  g_L(u) )^T =  \mathcal{AB} \ ( h_1(du+\widetilde{s}), ... ,  h_L(du+\widetilde{s}) )^T. 
\end{equation}
By construction, we see that
$g_j$ and $g''_j $ are polynomials in $\mathbb{F}_q[t][u]$ with coefficients supported
on a subset of $\mathcal{S}(\mathcal{K})$ and 
a subset of $\mathcal{S}(\mathcal{K}) \backslash \mathcal{H}$, respectively.
Then $(4)$ of this lemma follows by the fact that $[g_i]_{T_j} =[g'_i]_{T_j} \ (1 \leq i,j \leq L).$

Let
$$
(\gamma_1, ..., \gamma_L) = (\beta_1, ..., \beta_L) \ (\mathcal{A} \mathcal{B})^{-1}.
$$
Then we have
$$
\gamma_1 g_1(u) + ... + \gamma_L g_L(u) = \beta_1 h_1(du+\widetilde{s}) + ... + \beta_L h_L(du+\widetilde{s}).
$$

Finally, recall from above that the entries of matrices $\mathcal{A}$ and $\mathcal{B}$ are dependent only on $h_1, ..., h_L$.
Then $(3)$ of this lemma follows easily from  $(\ref{g to h})$.
\end{proof}

\section{Proof of the Main Results}
\label{proving}
We have collected enough material in the previous sections to prove our main results of the paper.
We begin this section by proving Theorems \ref{first theorem} and \ref{second theorem}.
\begin{proof}[Proof of Theorem \ref{first theorem}]
Let $\beta$ be an arbitrary element in $\mathbb{K}_{\infty}$.
Let $M = \lfloor \theta N \rfloor + 1$, where $\theta$ is a sufficiently small positive number to
be chosen later. We prove by contradiction that for any $N$ sufficiently large,
$$
\min_{x \in \mathbb{G}_N} \ord \{ \beta h(x) \}   \leq {-M} \leq  {- \theta N}.
$$
Suppose for some $N$ sufficiently large in terms of $\mathcal{K}$, $q$, $\theta$, and $\ord c_k$, we have
$\ord \{ \beta h(x) \}  > -M$ for all $x \in \mathbb{G}_N$.
Then by Lemma \ref{first lemma}, there exists $y \in \mathbb{G}_M \backslash \{ 0 \}$ such that
$$
\Big{|} \sum_{x \in \mathbb{G}_N}  e(y \beta h(x)) \Big{|} \geq \frac{q^N}{q^M-1}> q^{N-M}.
$$
It follows by Corollary \ref{Cor LL main} that for $\theta < c$ there exists $g \in \mathbb{F}_q[t]$ such that
$\ord g < C M$ and $\ord \{ g y \beta c_k \} \leq -k N +  C M$
for some constants $c, C>0$, depending only on $\mathcal{K}$ and $q$.
By the hypothesis, we know there exists $x \in \mathbb{G}_{\ord (g y c_k)}$
such that $h(x)\equiv 0 \ (\text{mod }g y c_k)$. Since $N$ is sufficiently large, by taking $\theta < 1 / (C+1)$ we have
$$
\ord x < \ord (g y c_k) < CM + M + \ord c_k  \leq N.
$$
We denote by $D$ some constant
dependent only on $h$. We have
\begin{eqnarray}
\ord \{ \beta h(x) \}  &\leq&
\min_{z \in \mathbb{F}_q[t]} \ord  \left( \beta h(x) - \frac{h(x)}{g y c_k} z  \right)
\notag
\\
&=& \ord \left( \frac{h(x)}{g y c_k} \right) + \ord \{ g y \beta c_k \}
\notag
\\
&\leq&  D + (\ord g + \ord y)( \deg h - 1 )  +  \ord \{ g y \beta c_k \}
\notag
\\
&\leq& D + (C M + M ) (\deg h - 1) + C M - kN
\notag
\\
&=&  D +  ((C  + 1 ) (\deg h - 1) + C ) M - kN.
\notag
\end{eqnarray}
Suppose
$$
\theta < \min \Big{\{} \frac{1}{ (C  + 1 ) (\deg h - 1) + C + 1}, \frac{1}{C+1} \Big{ \} }.
$$
Then for $N$ sufficiently large in terms of $D$, 
we obtain from above that $ \ord \{ \beta h(x) \} \leq -M$, which is a contradiction.
\end{proof}


\begin{proof}[Proof of Theorem \ref{second theorem}]
Let $\beta_1, ..., \beta_L$ be arbitrary elements in $\mathbb{K}_{\infty}$.
Let $M = \lfloor \theta N \rfloor + 1$
and $\theta$ be a sufficiently small positive number to be
chosen later. To obtain contradiction, suppose
for some $N$ sufficiently large in terms of $\mathcal{K}$, $q$ and $\theta$, we have
$$
\ord \{ \beta_1 h_1(x) + ... + \beta_L h_L(x) \} > -M
$$
for all $x \in \mathbb{G}_N$.

Let $\mathcal{T}$ and $g_1, ..., g_L$ be the matrix and polynomials, respectively, obtained by applying
Lemma \ref{cor to matrix lemma} to the polynomials $h_1, ..., h_L$.
We also have by $(3)$ and $(4)$ of Lemma \ref{cor to matrix lemma} that
there exist $T_j \in \mathcal{K}^* \ (1 \leq j \leq L)$ such that $[g_i]_{T_j} = 0$ if $i \not = j$,
and $\gamma_j \in \mathbb{K}_{\infty} \ (1 \leq j \leq L)$ such that
$$
\beta_1 h_1(u) + ... + \beta_L h_L(u) = \gamma_1 g_1(u) + ... + \gamma_L g_L(u).
$$
Hence for all $x \in \mathbb{G}_N$,
we have
\begin{equation}
\label{contradiction in second thm}
\ord \{ \gamma_1 g_1(x) + ... + \gamma_L g_L(x) \}   > -M.
\end{equation}
Then, by Lemma \ref{first lemma}, there exists $y \in \mathbb{G}_M \backslash \{ 0 \}$ with
$$
\Big{|} \sum_{x \in \mathbb{G}_N}  e(y \gamma_1 g_1(x) + ... + y \gamma_L g_L(x)) \Big{|} \geq \frac{q^N}{q^M-1}> q^{N-M}.
$$
Let $f(u) = y \gamma_1 g_1(u) + ... + y \gamma_L g_L(u)$, and suppose it is supported on $\mathcal{\widehat{K}} \subseteq \mathbb{Z}^+$.
We can verify that each $T_j \in (\mathcal{\widehat{K}})^*$.
Applying Corollary \ref{Cor LL main} with $f(u)$, 
we obtain that for $\theta < c$
there exists $g \in \mathbb{F}_q[t]$ such that
$\ord g < C M$ and
\begin{equation}
\label{inequality in second thm}
\ord \{ g [f]_{T_j} \} = \ord \{ g y \gamma_j [g_j]_{T_j} \} \leq C M - T_j N \ (1 \leq j \leq L),
\end{equation}
for some constants $c, C > 0$ depending only on $\mathcal{K}$ and $q$.

Let $v = g y \prod_{j=1}^L [g_j]_{T_j}$ and let $D$ be some constant dependent only
on $g_1, ..., g_L$.
Consequently, $D$ is dependent only on $h_1, ..., h_L$.
Note the actual value of $D$ may vary from line to line during calculations.
Then $$ \ord  v  \leq D + \ord g y  \leq D + C M + M,$$
and we make sure $\ord v < N$ by taking $N$ sufficiently large with respect to $D$.
Thus for all $1 \leq j \leq L$, we have
\begin{eqnarray}
\notag
\ord \{ v \gamma_j \}   &=& \min_{ z \in \mathbb{F}_q[t] } \ord (  v \gamma_j - z )
\\
&\leq& \min_{ z \in \mathbb{F}_q[t] } \ord \left( v \gamma_j -  z \prod_{i \not = j}[g_i]_{T_i} \right)
\notag
\\
&=& \min_{ z \in \mathbb{F}_q[t] } \ord \left( \prod_{i \not = j}[g_i]_{T_i} \right) +  \ord \left( g y \gamma_j [g_j]_{T_j} -  z \right)
\notag
\\
&=& \ord \left( \prod_{i \not = j}[g_i]_{T_i} \right) + \min_{ z \in \mathbb{F}_q[t] }  \ord \left( g y \gamma_j [g_j]_{T_j} -  z \right)
\notag
\\
&=& \sum_{i \not = j} \ord \left( [g_i]_{T_i} \right) + \ord \{  g y \gamma_j [g_j]_{T_j} \}
\notag
\\
&\leq& D + C M  - T_j N,
\label{lots of inequalities 1}
\end{eqnarray}
where we used ~(\ref{inequality in second thm}) to obtain the last inequality.
It follows that if we
let $a_j = ( v \gamma_j - \{ v \gamma_j \} ) \in \mathbb{F}_q[t]$, then we have
$$ \ord \left( \gamma_j - \frac{a_j}{v} \right) \leq D + C M - T_j N - \ord v \ \ (1 \leq j \leq L).$$
Recall each $g_j$ is a linear combination over $\mathbb{F}_q[t]$ of $h_1, ..., h_L$.
Thus by the hypothesis, we know there exists $n \in \mathbb{G}_{\ord v}$ such that
$$
a_1 g_1(n) + ... + a_L g_L(n) \equiv 0 \ (\text{mod } v).
$$
Clearly, we have $\max_{ 1\leq j \leq L} \deg g_j \leq \max_{ 1\leq j \leq L} \deg h_j$. Thus we obtain
\begin{eqnarray}
\ord \Big{ \{ } \sum_{j=1}^L \gamma_j g_j(n) \Big{ \} }
&\leq&
\ord \left( \sum_{j=1}^L \gamma_j g_j(n) - \frac{1}{v}  \sum_{j=1}^L a_j g_j(n) \right)
\notag
\\
&=&
\ord \left( \sum_{j=1}^L \left( \gamma_j - \frac{a_j}{v}  \right) g_j(n) \right)
\notag
\\
&\leq&
\max_{ 1\leq j \leq L} \ord \left( \left( \gamma_j - \frac{a_j}{v}  \right) g_j(n) \right)
\notag
\\
&\leq&
\max_{ 1\leq j \leq L} C M - T_j N + (\ord v)( \deg g_j - 1) + D
\notag
\\
&\leq&
\max_{ 1\leq j \leq L} C M - T_j N + (CM + M)( \deg g_j - 1) + D
\notag
\\
&\leq&
-N  +   C M + (C + 1) \left( \max_{ 1\leq j \leq L} \deg h_j - 1 \right) M + D.
\notag
\end{eqnarray}
Suppose $\theta$ is sufficiently small in terms of $C$ and $\max_{ 1\leq j \leq L} \deg h_j$.
Then it is not too difficult to see that the final quantity obtained above is
less than or equal to $-M$ for $N$ sufficiently large, which
contradicts ~(\ref{contradiction in second thm}).
Therefore, there exists some $m \in \mathbb{G}_N$ such that
$$
\ord \{ \beta_1 h_1(m) + ... + \beta_L h_L(m)  \} \leq -M \leq - \theta N.
$$
\end{proof}

Recall from Section \ref{prelim} that corresponding to any system of polynomials
$(h_1, ..., h_L)$ satisfying Condition $ ( \star ) $, we can associate a sequence $(r_x)_{ x \in \mathbb{F}_q[t] \backslash \{ 0 \} }$
such that ~(\ref{r_d's}) is satisfied. 
We prove the following theorem.
\begin{thm}
\label{third theorem}
Let $h_j \in \mathbb{F}_{q}[t] [u]$ be supported on a set $\mathcal{K}_j \subseteq \mathbb{Z}^+ \ (1 \leq j \leq L)$, and
let $\mathcal{K} = \mathcal{K}_1 \cup ... \cup \mathcal{K}_L$.
Suppose the system $( h_j )_{j=1}^L$ satisfies Condition $ ( \star ) $ and that
the maximal $\mathcal{K}^*$-portion of $( h_j )_{j=1}^L$ is linearly independent.
Then there exist $\theta = \theta(\mathcal{K},q, \max_{1 \leq j \leq L} \deg h_j ), \sigma = \sigma(h_1, ..., h_L) > 0$ and
$N_0 = N_0(\mathcal{K}, q, \theta, \sigma, h_1, ..., h_L) \in \mathbb{Z}^+$ such that the following holds when $N > N_0$.
Given any $d \in \mathbb{F}_q[t]$ with $\ord d < \lfloor \sigma N \rfloor$, and $\beta_1, ..., \beta_L$ in $\mathbb{K}_{\infty}$,
there exists $n \in \mathbb{G}_N$ such that
$n \equiv r_d \ (\text{mod } d)$ and
$$
\ord \{ \beta_1 h_1(n) + ... + \beta_L h_L(n) \} \leq -\theta N.
$$
\end{thm}

Before we prove the theorem, we give an example of a system $(h_1, h_2) \subseteq \mathbb{F}_5[t][u]$ that satisfies
the hypothesis of Theorem \ref{third theorem}, where $h_1(u)$ and $h_2(u)$ do not share a common root in $\mathbb{F}_5[t]$.

\begin{example}
\label{example three}
Let $p=5$, and consider $h_1(u) = (u+1)(u^2 - t)(u^2 - (t+1))(u^2 - t(t+1))$ and $h_2(u) = u^{20} (u^2 - t)(u^2 - (t+1))(u^2 - t(t+1))$
in $\mathbb{F}_5[t][u]$. It is clear that $h_1(u)$ and $h_2(u)$ do not share a common root in $\mathbb{F}_5[t]$.
It follows from Example \ref{example two} and Lemma \ref{joint to single} that the system $(h_1, h_2)$
satisfies Condition $ ( \star ) $. 
The polynomials $h_1(u)$ and $h_2(u)$ are supported on $\mathcal{K}_1 = \{ 7, 6, 5, 4, 3, 2, 1 \}$
and $\mathcal{K}_2 = \{ 26, 24, 22, 20 \}$, respectively. We can verify that
$$
h_1^{\text{max }}(u) = u^7 \ \ \text{  and  }  \ \  h_2^{\text{max }}(u) = u^{26} - (t^2 + 3t + 1) u^{24},
$$
which are clearly linearly independent over $\mathbb{K}$.
Therefore, the maximal $\mathcal{K}^*$-portion of $(h_1,h_2)$ is linearly independent, and it follows that
$(h_1, h_2)$ satisfies the hypothesis of Theorem \ref{third theorem}. It is also easy to see that $(h_1,h_2)$
satisfies the hypothesis of Theorem \ref{second theorem}.
\end{example}

\begin{proof}
Let $\beta_1, ..., \beta_L$ be arbitrary elements in $\mathbb{K}_{\infty}$ and
$d$ be an arbitrary element in $\mathbb{G}_{\lfloor \sigma N \rfloor}$.
Let $\theta$ and $\sigma$ be positive real numbers sufficiently small to be chosen later,
and let $M = \lfloor \theta N \rfloor + 1$.
Suppose for some $N$ sufficiently large in terms of $\mathcal{K}$, $q$, $\sigma$, and $\theta$, we have
$$
\ord \{  \beta_1 h_1(d x + r_d) + ... +  \beta_L h_L(d x + r_d) \} > - M
$$
for all $x \in \mathbb{G}_{\lfloor (1 - \sigma) N \rfloor}$.
Let $\mathcal{T}$ and $g_1, ..., g_L$ be the matrix and polynomials, respectively, obtained by applying
Lemma \ref{second cor to matrix lemma} to the polynomials $h_1, ..., h_L$ with $s = r_d$ and $d$.
By $(4)$ of Lemma \ref{second cor to matrix lemma}, we know there exist $T_j \in \mathcal{K}^* \ (1 \leq j \leq L)$
such that $T_j$ is maximal in $\mathcal{K}$, $[g_i]_{T_j} = 0$ if $i \not = j$,
and $[g_j]_{T_j} = \tilde{c}_j d^{T_j}$ for some $\tilde{c}_j \in \mathbb{F}_q[t]$
dependent only on $h_1, ..., h_L$.
We also know there exist $\gamma_j \in \mathbb{K}_{\infty} \ (1 \leq j \leq L)$ such that
$$
\beta_1 h_1(d u + r_d) + ... + \beta_L h_L(d u + r_d) = \gamma_1 g_1(u) + ... + \gamma_L g_L(u).
$$
Thus we have
\begin{equation}
\label{contradiction in third thm}
\ord \{  \gamma_1 g_1(x) + ... +  \gamma_L g_L(x) \} > - M
\end{equation}
for all $x \in \mathbb{G}_{\lfloor (1 - \sigma) N \rfloor}.$
By Lemma $\ref{first lemma}$, there exists $y \in \mathbb{G}_{M} \backslash \{ 0 \}$ such that
$$
\Big{|} \sum_{x \in \mathbb{G}_{\lfloor (1 - \sigma) N \rfloor}  }  e(y \gamma_1 g_1(x) + ... + y \gamma_L g_L(x)) \Big{|}
\geq \frac{q^{\lfloor (1 - \sigma) N \rfloor}}{q^{ M}-1}
>
q^{N - (\sigma + \theta)N }.
$$
Let $f(u) = y \gamma_1 g_1(u) + ... + y \gamma_L g_L(u)$, and suppose it is supported on $\mathcal{\widehat{K}} \subseteq \mathbb{Z}^+$.
We can verify that each $T_j \in (\mathcal{\widehat{K}})^*$.
Applying Corollary \ref{Cor LL main} with $f(u)$, 
we obtain that for $(\sigma + \theta) < c$
there exists $g \in \mathbb{F}_q[t]$ such that
$\ord g < C (\sigma + \theta) N$ and
\begin{equation}
\label{inequality of third theorem}
\ord \{ g [f]_{T_j} \} = \ord \{ g y \gamma_j [g_j]_{T_j} \} \leq C (\sigma + \theta) N - T_j \lfloor (1 - \sigma) N \rfloor \ \ (1 \leq j \leq L),
\end{equation}
for some constants $c, C > 0$ depending only on $\mathcal{K}$ and $q$.
Let $v = g y \prod_{j=1}^L [g_j]_{T_j}$ and let $D$ be some constant dependent only
on $h_1, ..., h_L$ (note the actual value of $D$ may vary from line to line during
calculations). We define $T' = \sum_{1 \leq j \leq L } T_j.$
Then
$$
\ord v  \leq \ord  g y  + T' \ord d + D \leq C (\sigma + \theta) N + M + T' \lfloor \sigma N \rfloor + D.
$$
In particular, we have $\ord v < \lfloor (1 - \sigma) N \rfloor$ for $N$ sufficiently large with respect to $D$
and $\theta, \sigma$ sufficiently small.

For simplicity denote $n = r_v \in \mathbb{G}_{ \ord v}$,
then $h_j(n)$ is divisible by $v$ for any $1 \leq j \leq L$.
We also have $n \equiv r_d \ (\text{mod }d)$, because $d | v$.
Each $g_j(u)$ can be written
as an $\mathbb{F}_q[t]$-linear combination of the polynomials $h_1(d u + r_d), ..., h_L(d u + r_d)$. Thus if we write $n = d w + r_d$
for some $w \in \mathbb{G}_{\ord v}$, then $g_j(w)$ is divisible by $v$
for any $1 \leq j \leq L$.
Let $H = \max_{1 \leq j \leq L} \deg h_j$.
Then it follows that
\begin{eqnarray}
\notag
\ord \{  \gamma_j g_j(w)  \} &\leq& \min_{z \in \mathbb{F}_q[t]} \ord \left( \gamma_j g_j(w) - z \frac{g_j(w)}{v} \right)
\\
&=& \ord \left( \frac{g_j(w)}{v} \right) + \min_{z \in \mathbb{F}_q[t]} \ord \left( v \gamma_j - z \right)
\notag
\\
&=&
\ord \left( \frac{g_j(w)}{v} \right) + \ord \{ v \gamma_j \}
\notag
\\
&\leq&
D + H (\ord d + \ord w) -\ord v + \ord \{ v \gamma_j \},
\label{lots of inequalities 2}
\end{eqnarray}
where the last inequality is obtained via $(3)$ of Lemma \ref{second cor to matrix lemma}.
We also have by similar calculations as in ~(\ref{lots of inequalities 1}) that for $1 \leq j \leq L$,
\begin{eqnarray}
\notag
\ord \{ v \gamma_j \}   &=& \min_{ z \in \mathbb{F}_q[t] } \ord (  v \gamma_j - z )
\\
&\leq& \min_{ z \in \mathbb{F}_q[t] } \ord \left( v \gamma_j -  z \prod_{i \not = j}[g_i]_{T_i} \right)
\notag
\\
&=& \ord \left( \prod_{i \not = j}[g_i]_{T_i} \right) + \min_{ z \in \mathbb{F}_q[t] }  \ord \left( g y \gamma_j [g_j]_{T_j} -  z \right)
\notag
\\
&=& \sum_{i \not = j} \ord \left( [g_i]_{T_i} \right) + \ord \{  g y \gamma_j [g_j]_{T_j} \}
\notag
\\
&=& \left( \sum_{i \not = j} \ord \tilde{c}_i + T_i \ \ord d  \right) + \ord \{  g y \gamma_j [g_j]_{T_j} \}
\notag
\\
&\leq& T' \ord d + D + C (\sigma + \theta) N  - T_j \lfloor (1 - \sigma) N \rfloor,
\label{lots of inequalities 3}
\end{eqnarray}
where we used ~(\ref{inequality of third theorem}) to obtain the last inequality.
Therefore, we have by ~(\ref{ord fractional part}), ~(\ref{lots of inequalities 2}), and ~(\ref{lots of inequalities 3}) that
\begin{eqnarray}
&&\ord \Big{ \{ } \sum_{j=1}^L \gamma_j g_j(w) \Big{ \} }
\notag
\\
&\leq&
\max_{1 \leq j \leq L} \ord \{  \gamma_j g_j(w)  \}
\notag
\\
&\leq&
T' \ord d + D + C (\sigma + \theta) N  - \lfloor (1 - \sigma) N \rfloor \min_{ 1\leq j \leq L} T_j + H (\ord d + \ord w) -\ord v
\notag
\\
&\leq&
\sigma T' N + D + C (\sigma + \theta) N  - \lfloor (1 - \sigma) N \rfloor \min_{ 1\leq j \leq L} T_j + H (\sigma N + \ord v) - \ord v
\notag
\\
&\leq&
\sigma T' N + D + C (\sigma + \theta) N  - \lfloor (1 - \sigma) N \rfloor +  \sigma H N + H ( C(\sigma + \theta) N + M + \sigma T' N).
\notag
\end{eqnarray}
Suppose $\theta$ is sufficiently small
in terms of $C$ and $H$, and also that $\sigma$ is sufficiently small
in terms of $C$, $T'$ and $H$. Then for $N$ sufficiently large, the final quantity
obtained above is less than or equal to $- M$, which contradicts ~(\ref{contradiction in third thm}).
Therefore, there exists $x \in \mathbb{G}_{\lfloor (1 - \sigma) N \rfloor}$ such that
$ m = d x + r_d \in \mathbb{G}_{N}$ and
$$
\ord \{ \beta_1 h_1(m) + ... + \beta_L h_L(m)  \} \leq -M \leq - \theta N.
$$
\end{proof}

\appendix
\section{} \label{A}

In this section, we prove the statement presented in Example \ref{example two}.
We let $p=5$ and let $h(u) = (u^2 - t)(u^2 - (t+1)) (u^2 - (t^2+t)) \in \mathbb{F}_{5}[t][u]$.
First, since $\mathbb{F}_{5}[t]$ is a unique factorization domain, it is clear that $h(u)$
has a root in $\mathbb{F}_{5}[t]$ if and only if at least one of $u^2 - t, u^2 - (t+1)$, and  $u^2 - (t^2+t)$
has a root in $\mathbb{F}_{5}[t]$. It can be verified easily that none of the three polynomials have
a root in $\mathbb{F}_{5}[t]$. Therefore, $h(u)$ does not have a root in $\mathbb{F}_{5}[t]$.

In order to prove that $h(u)$ has a root modulo $g$ for every $g$ in $\mathbb{F}_5[t] \backslash \{ 0\}$,
we use the following version of the Hensel's lemma.
\begin{lem}\textnormal{\cite[Lemma 25]{RMK}} 
Suppose $\text{char }({\mathbb{F}}_q) = p \nmid K$, and we have $w, z_0, z \in \mathbb{F}_q[t]$, where $w$ is irreducible, $w \nmid z_0$, and
$$
z^K \equiv z_0 \ (\text{mod } w^A).
$$
Then, for $B>A$, there exists $z' \in \mathbb{F}_q[t]$ such that
$$
(z')^K \equiv z_0 \ (\text{mod } w^B) \ \ \text{ and } \ \  z' \equiv z \ (\text{mod } w^A).
$$
\end{lem}
By the quadratic reciprocity law in $\mathbb{F}_5[t]$, given any irreducible $\pi \in \mathbb{F}_5[t]$ that is not $t$ or $t+1$,
we know that either $t^2+t$ is a quadratic residue modulo $\pi$, or one of $t$ and $t+1$ is a quadratic residue modulo $\pi$.
Suppose $\pi = t+1$, then we have
$$
2^2 \equiv -1 \equiv t \ (\text{mod }\pi).
$$
On the other hand, if $\pi = t$, then we have
$$
1^2 \equiv 1 \equiv t+1 \ (\text{mod }\pi).
$$
By the Hensel's lemma above, it follows that given any $L \geq 1$, one of $u^2 - t$,  $u^2 - (t+1)$,
and $u^2 - (t^2 + t)$ has a root modulo $\pi^L$. In other words, $h(u)$ has a root modulo $\pi^L$.
It then follows from the Chinese Remainder Theorem that $h(u)$ has a root modulo $g$ for every $g$ in $\mathbb{F}_5[t] \backslash \{ 0\}$.

\section*{Acknowledgement} This paper was part of my PhD thesis. I would like to thank Professor Yu-Ru Liu, Professor Th\'{a}i Ho\`{a}ng L\^{e}, and 
Professor Craig V. Spencer for suggesting this problem to me. I would also like to thank the referee for many helpful comments and suggestions.

\end{document}